\RequirePackage[l2tabu, orthodox]{nag}
\pdfoutput=1
\documentclass[a4paper,12pt]{article}
\usepackage{fullpage}
\usepackage{mathrsfs}
\usepackage{datetime}
\longdate
\usepackage{lmodern}
\usepackage{mparhack}
\usepackage{amsmath, amssymb, amsfonts, amsthm, mathtools}
\usepackage{fixmath}
\mathtoolsset{centercolon}
\usepackage[utf8]{inputenc}
\usepackage[T1]{fontenc}
\usepackage{graphicx}
\usepackage{enumerate}
\usepackage{tikz}
\usetikzlibrary{arrows.meta}
\usepackage{caption}
\usepackage{subcaption}
\usepackage{hyperref}

\theoremstyle{plain}
\newtheorem{theorem}{Theorem}
\newtheorem{lemma}[theorem]{Lemma}
\newtheorem{corollary}[theorem]{Corollary}

\newtheorem{proposition}[theorem]{Proposition}

\theoremstyle{definition}
\newtheorem{definition}[theorem]{Definition}

\newtheorem{question}[theorem]{Question}

\newcommand{\setbuilder}[2]{\left\{#1\;\middle|\;#2\right\}}
\newcommand{\set}[1]{\left\{#1\right\}}

\newcommand{\epsi}{\varepsilon}
\newcommand{\fhi}{\varphi}
\newcommand{\ipr}[2]{\left\langle #1, #2 \right\rangle}
\newcommand{\numbersystem}[1]{\mathbb{#1}}
\newcommand{\bN}{\numbersystem{N}}
\newcommand{\bR}{\numbersystem{R}}
\newcommand{\bZ}{\numbersystem{Z}}
\DeclareMathOperator{\conv}{conv}
\DeclareMathOperator{\bd}{bd}

\newcommand{\Red}{\bR^d}

\newcommand{\csep}{c_{\textnormal{sep}}} 
\newcommand{\floor}[1]{\left\lfloor#1\right\rfloor}
\newcommand{\norm}[2][]{\left\lVert#2\right\rVert_{#1}}
\newcommand{\Hsep}{H_{\textnormal{sep}}}

\newcommand{\arc}[1]{\overarc{#1}}
\newcommand{\darc}[1]{\overarcarrow{#1}}
\newcommand{\overarcarrow}[1]{\ifmmode\mathchoice{%
\tikz [baseline = (N.base)] {
      \node [inner sep = 0pt] (N) {$\displaystyle#1$};
      \draw [line width = 0.4pt,-{Stealth[scale=0.6]}] 
         ([xshift=0.1ex,yshift=0.3ex]N.north west) to[out=15,in=165]
         ([yshift=0.3ex]N.north east);
   }
}{%
\tikz [baseline = (N.base)] {
      \node [inner sep = 0pt] (N) {$\textstyle#1$};
      \draw [line width = 0.4pt,-{Stealth[scale=0.6]}] 
         ([xshift=0.1ex,yshift=0.3ex]N.north west) to[out=15,in=165]
         ([yshift=0.3ex]N.north east);
   }
}{%
\tikz [baseline = (N.base)] {
      \node [inner sep = 0pt] (N) {$\scriptstyle#1$};
      \draw [line width = 0.3pt,-{Stealth[scale=0.6]}] 
         ([xshift=0.1ex,yshift=0.3ex]N.north west) to[out=15,in=165]
         ([yshift=0.3ex]N.north east);
   }
}{%
\tikz [baseline = (N.base)] {
      \node [inner sep = 0pt] (N) {$\scriptscriptstyle#1$};
      \draw [line width = 0.213pt,-{Stealth[scale=0.6]}] 
         ([xshift=0.1ex,yshift=0.3ex]N.north west) to[out=15,in=165]
         ([yshift=0.3ex]N.north east);
   }
}   
\else   
\tikz [baseline = (N.base)] {
      \node [inner sep = 0pt] (N) {#1};
      \draw [line width = 0.4pt,-{Stealth[scale=0.6]}] 
         ([xshift=0.1ex,yshift=0.3ex]N.north west) to[out=15,in=165]
         ([yshift=0.3ex]N.north east);
   }
\fi   
}
\newcommand{\overarc}[1]{\ifmmode\mathchoice{%
\tikz [baseline = (N.base)] {
      \node [inner sep = 0pt] (N) {$\displaystyle#1$};
      \draw [line width = 0.4pt,] 
         ([xshift=0.1ex,yshift=0.3ex]N.north west) to[out=15,in=165]
         ([yshift=0.3ex]N.north east);
   }
}{%
\tikz [baseline = (N.base)] {
      \node [inner sep = 0pt] (N) {$\textstyle#1$};
      \draw [line width = 0.4pt,] 
         ([xshift=0.1ex,yshift=0.3ex]N.north west) to[out=15,in=165]
         ([yshift=0.3ex]N.north east);
   }
}{%
\tikz [baseline = (N.base)] {
      \node [inner sep = 0pt] (N) {$\scriptstyle#1$};
      \draw [line width = 0.3pt,] 
         ([xshift=0.1ex,yshift=0.3ex]N.north west) to[out=15,in=165]
         ([yshift=0.3ex]N.north east);
   }
}{%
\tikz [baseline = (N.base)] {
      \node [inner sep = 0pt] (N) {$\scriptscriptstyle#1$};
      \draw [line width = 0.213pt,] 
         ([xshift=0.1ex,yshift=0.3ex]N.north west) to[out=15,in=165]
         ([yshift=0.3ex]N.north east);
   }
}   
\else   
\tikz [baseline = (N.base)] {
      \node [inner sep = 0pt] (N) {#1};
      \draw [line width = 0.4pt,] 
         ([xshift=0.1ex,yshift=0.3ex]N.north west) to[out=15,in=165]
         ([yshift=0.3ex]N.north east);
   }
\fi   
}

\newcommand{\noshow}[1]{}

\author{M\'arton Nasz\'odi\footnote{
Lor\'and E\"otv\"os University and
MTA-ELTE Lend\"ulet Combinatorial Geometry Research Group, Budapest}
\and Konrad J. Swanepoel\footnote{Department of Mathematics, London School of 
Economics}}

\title{%
Contacts in totally separable packings in the plane and in high dimensions%
\footnote{Part of the research was carried out while MN was a member of 
J\'anos Pach's chair of DCG at EPFL, Lausanne, which was supported by Swiss 
National Science Foundation Grants 200020-162884 and 200021-175977, and while 
MN and KS visited the Mathematical Research Institute Oberwolfach in their 
Research in Pairs programme.
MN also acknowledges the support of the National Research, Development and 
Innovation Fund (NRDI) grants K119670 and KKP-133864, the Bolyai Scholarship of the Hungarian Academy of Sciences and the New National Excellence and the TKP2020-NKA-06 Programs provided by the NRDI.}}\date{}

\begin{document}
\maketitle
\begin{abstract}
We study the contact structure of \emph{totally separable} packings of 
translates of a convex body $K$ in $\bR^d$, that is, packings where any 
two touching bodies have a separating hyperplane that does not intersect 
the interior of any translate in the packing.
The \emph{separable Hadwiger number} $\Hsep(K)$ of $K$ is defined to be the maximum number of translates touched by a single translate, with the maximum taken over all totally separable packings of translates of $K$.
We show that for each $d\geq 8$, there exists a smooth and strictly convex $K$ in $\bR^d$ with $\Hsep(K)>2d$, and asymptotically, $\Hsep(K)=\upOmega\bigl((3/\sqrt{8})^d\bigr)$.

We show that Alon's packing of Euclidean unit balls such that each translate touches at least $2^{\sqrt{d}}$ others whenever $d$ is a power of~$4$, can be adapted to give a totally separable packing of translates of the $\ell_1$-unit ball with the same touching property.

We also consider the maximum number of touching pairs in a 
totally separable packing of $n$ translates of any planar convex body~$K$.
We prove that the maximum equals $\floor{2n-2\sqrt{n}}$ if and only if $K$ is a quasi hexagon, thus completing the determination of this value for all planar convex bodies.
\end{abstract}

\section{Introduction}
We define a \emph{translative packing} of a convex body $K$ in $\bR^d$ to be a collection $\{x_1+K,x_2+K,\dots,x_n+K\}$ of translates of $K$ 
such that no two translates have interior points in common.
We say that two such translates \emph{touch} if they have a common boundary point.
A packing is \emph{totally separable} if any two touching bodies have 
a separating hyperplane that does not intersect the interior of any translate 
in the packing. The \emph{contact graph} of the packing is the graph defined on the set of translates in the packing where two translates are joined by an edge if they touch.

In this paper, we study the maximum and minimum degrees, and the number of 
edges of contact graphs of totally separable packings of a convex body.

The \emph{separable Hadwiger number}, denoted as $\Hsep(K)$, of a convex body 
$K$ in $\Red$ is the maximum degree in the contact graph of any totally 
separable packing of $K$.

This number is known for the $d$-cube $[0,1]^d$.
Any contact graph of a translative packing of the $d$-cube 
can be embedded in the contact graph of a $\bZ^d$-lattice packing of the cubes (Fejes 
T\'oth and Sauer \cite{FTS77}), which is clearly totally separable.
Therefore, the maximum possible degree equals the Hadwiger number of the 
$d$-cube, which is $3^d-1$.
It is also known that the maximum possible minimum degree equals 
$(3^d-1)/2$ (Talata \cite{T11}).

For arbitrary $K$ we have the well-known bounds $2d\leq\Hsep(K)\leq 3^d-1$.
If $K$ is smooth or strictly convex and $d\leq 4$, then $\Hsep(K)=2d$ \cite{BN18}.
In contrast, we show that for all $d\geq 8$ there exists a 
smooth and strictly convex body $K$
in $\bR^d$ with $\Hsep(K)>2d$.
\begin{theorem}\label{thm:ddimsmall}
For each $d\geq 8$ there exists a $d$-dimensional smooth and strictly convex body with a separable 
Hadwiger number of at least $2d+2$.
For $d=9, 10,11,12,14$, there is a lower bound of $29, 39, 50, 65, 91$, respectively.
\end{theorem}
These lower bounds are obtained from the existence of certain spherical codes, obtained by Conway--Hardin--Sloane \cite{CHS96}.

We find some upper bounds for dimensions $5$, $6$, and $7$.

\begin{proposition}\label{thm:uppersmall2}
The separable Hadwiger number of a $5$, $6$, $7$-dimensional smooth and strictly convex body is at most $15$, $27$, $63$, respectively.
\end{proposition}
For large $d$ we have an exponential lower bound.
\begin{theorem}\label{thm:ddimbig}
There exists a $d$-dimensional smooth and strictly convex body with a separable 
Hadwiger number of at least $\upOmega\bigl((3/\sqrt{8})^d\bigr)$.
\end{theorem}

We show that the $\ell_1$ ball has a totally separable packing with high 
minimum degree in the contact graph.
\begin{theorem}\label{thm:ell1}
There exists a finite totally separable packing of $\ell_1^d$-unit balls such 
that each ball touches at least $2^{\sqrt{d}}$ others, when $d$ is a power of 
$4$. 
\end{theorem}

The number of edges in the contact graph of a packing, that is, the number of 
touching pairs is called the \emph{contact number} of the packing. For a 
positive integer $n$, we denote the largest contact number of a totally 
separable packing of $n$ translates of $K$ by $\csep(K,n)$.

The following class of $o$-symmetric convex bodies in the plane turns out to be 
special from the point of view of totally separable packings. We call $K$ 
a \emph{quasi hexagon}, if $K$ contains two line segments on its boundary with a common 
endpoint, such that both are of at least unit length in the norm with unit 
ball~$K$ (Figure~\ref{fig:quasi}).
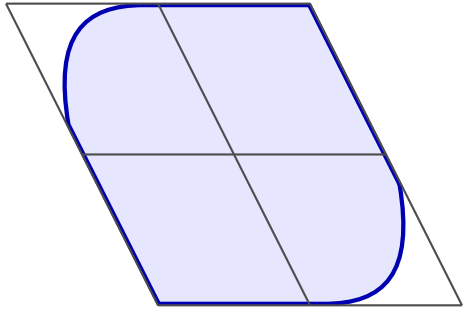
\begin{figure}
\centering
  \begin{tikzpicture}[thick,black!70!white,scale=2,cm={1,0,-0.5,1,(0,0)}]
  \draw[ultra thick, draw=blue!70!black, fill=blue!10!white, scale = 0.99]
	(1,-0.2) -- (1,1) -- (-0.1,1) .. controls (-0.5,1) and (-0.8,0.8) ..
	(-1,0.2) -- (-1,-1) -- (0.1,-1) .. controls (0.5,-1) and (0.8,-0.8) .. cycle;
  \draw (-1,-1) -- +(0,2);
  \draw (0,-1) -- +(0,2);
  \draw (1,-1) -- +(0,2);
  \draw (-1,-1) -- +(2,0);
  \draw (-1,0) -- +(2,0);
  \draw (-1,1) -- +(2,0);
  \end{tikzpicture}
\caption{A quasi hexagon}\label{fig:quasi}
\end{figure}
For example, parallelograms and affine regular hexagons are quasi hexagons.
See Lemma~\ref{lem:quasi} for a characterization.
We show the following upper bound for $\csep(K,n)$ if $K$ is not a quasi hexagon.
\begin{theorem}\label{thm:csep}
 Let $K$ be an $o$-symmetric convex body on the plane.
 If $K$ is not a quasi hexagon, then $\csep(K,n)\leq\floor{2n-2\sqrt{n}}$.
\end{theorem}
The special case when $K$ is the Euclidean disc 
was solved by Bezdek, Szalkai and Szalkai \cite{BSS16} using Harborth's method 
\cite{H74}. 
This result was extended to smooth and strictly convex bodies by Bezdek, Khan 
and Oliwa \cite{BKO} using a type of angular measure called B-measures.
However, there are many planar norms without B-measures \cite{NPS20}, and we follow an approach avoiding them entirely.
As a corollary, the exact values of $\csep(K,n)$ are now known for every planar convex body $K$.

\begin{corollary}\label{cor:csep}
 Let $K$ be an $o$-symmetric convex body on the plane.
 \begin{enumerate}
  \item If $K$ is a parallelogram, then $\csep(K,n)=\floor{4n-\sqrt{28n-12}}$.
  \item If $K$ is a quasi hexagon but not a parallelogram, then 
$\csep(K,n)=\floor{3n-\sqrt{12n-3}}$.
  \item If $K$ is not a quasi hexagon, then $\csep(K,n)=\floor{2n-2\sqrt{n}}$.
 \end{enumerate}
\end{corollary}

The upper bounds in the first two claims above follow directly from Brass' results in \cite{Br96}, while the lower bounds in all three claims follow from suitable lattice packings (Figure~\ref{fig:totsep}).

As pointed out in \cite[Remark~14]{BKO}, the proof of 
the lower bound $\csep(K,n)\geq\floor{2n-2\sqrt{n}}$ does not make use of 
smoothness. In fact, it relies only on the existence of a parallelogram $P$ 
containing $K$ with the property that the midpoints of the edges of $P$ are in 
$K$.
Any parallelogram that contains $K$ and of minimum area has this property.
The maximum number of pairs of cells that touch in an edge in a polyomino of $n$ cells is $\floor{2n-2\sqrt{n}}$ \cite{BKO}.

We prove Theorems~\ref{thm:ddimsmall}, \ref{thm:ddimbig} and \ref{thm:ell1} and Proposition~\ref{thm:uppersmall2} in 
Section~\ref{sec:ddim}, and Theorem~\ref{thm:csep} in Sections~\ref{sec:preliminaries} and \ref{sec:csep}. We conclude with some open problems in Section~\ref{sec:final}.

\section{Many contacts in dimension \texorpdfstring{$d$}{d}}\label{sec:ddim}
To prove Theorems~\ref{thm:ddimsmall} and \ref{thm:ddimbig}, we use the following result from \cite{BN18} to find smooth and strictly convex bodies with large separable Hadwiger number.
\begin{lemma}[{\cite[Lemma~2.1 and Note~2.2]{BN18}}]\label{lem:BN18}
There exists an $o$-symmetric, smooth and strictly convex body in $\bR^d$ with $\Hsep(K)\geq n$ if, and only if, there exist vectors $x_1,x_2,\dots,x_n\in\bR^d$ and linear functionals $\fhi_1,\fhi_2,\dots,\fhi_n\in(\bR^d)^*$ such that for all distinct~$i,j$,
\begin{itemize}
\item $\fhi_i(x_i)=1$ and $-1\leq \fhi_i(x_j)\leq 0$, and
\item $\fhi_i(x_j)>-1$ whenever $x_j\neq -x_i$ or $\fhi_j\neq -\fhi_i$. \qed
\end{itemize} 
\end{lemma}
\begin{proposition}\label{prop:exist}
Let $\alpha<1$ and suppose that there exist Euclidean unit vectors 
$v_1,\dots,v_m\in\bR^{d-1}$ such that $\ipr{v_i}{v_j} 
\in 
(-1+2\alpha,\alpha]$ for any distinct $i,j$.
Then for each $k\geq 0$ there exists a smooth and strictly convex body in $\bR^{d+k}$ with a 
totally separable Hadwiger number of at least $m+2k$.
\end{proposition}
\begin{proof}
For each $i=1,\dots,m$,
let $x_i = (v_i,1,o)\in\bR^{d-1}\oplus\bR\oplus\bR^k=\bR^{d+k}$ and
\[\fhi_i = \frac{1}{1-\alpha}(v_i,-\alpha,o)\in\left(\bR^{d-1}\right)^*\oplus\bR^*\oplus(\bR^k)^*=\left(\bR^{d+k}\right)^*,\]
and for each $j=1,\dots,k$, let $x_{m+j}=(o,0,e_j)$ and $\fhi_{m+j}=(o,0,e_j^*)$, where $\{e_1,\dots,e_k\}$ and $e_1^*,\dots,e_k^*$ are the standard orthonormal bases of $\bR^k$ and $(\bR^k)^*$.
Then $\fhi_i(x_i) = 1$ and $\fhi_i(x_j) = 
\frac{\ipr{v_i}{v_j}-\alpha}{1-\alpha}\in(-1,0]$ for all distinct $i,j\leq m$,
and if $i>m$ or $j>m$ then
\[ \fhi_i(x_j)=\begin{cases}
        1  &\iff i=j,\\
        -1 &\iff i,j>m \text{ and } |i-j|=d,\\
        0 &\text{ otherwise}.
    \end{cases}\]
It then follows from Lemma~\ref{lem:BN18} that there exists a 
smooth and strictly convex body $K\subseteq\bR^{d}$ such that 
$\setbuilder{x+K}{x=o\text{ or }x=x_i\text{ for some }i}$ is a totally 
separable packing of $m+1$ translates of $K$ with all $x_i+K$ touching $K$.
\end{proof}

\begin{proof}[Proof of Theorem~\ref{thm:ddimsmall}]
There exist $18$ unit vectors in $\bR^7$ with pairwise inner products in the interval $(-1/3,1/3)$ (Conway, Hardin, Sloane \cite{CHS96}).
This, together with Proposition~\ref{prop:exist}, gives the existence of a smooth and strictly convex $K$ in $\bR^{d}$ with $\Hsep(K)\geq 2d+2$ for each $d\geq 8$.
Similarly, for the other claims we observe from \cite{CHS96} that there exist $29, 39, 50, 65, 91$ such unit vectors in $\bR^8, \bR^9, \bR^{10}, \bR^{11},\bR^{13}$, respectively.
\end{proof}

\begin{proof}[Proof of Theorem~\ref{thm:ddimbig}]
By Proposition~\ref{prop:exist}, it is sufficient to find Euclidean unit vectors with 
pairwise inner products in $(-1/3,1/3)$.
 It is well known that there exist $\upOmega\bigl((3/\sqrt{8})^d\bigr)$ such unit vectors.
 Here we give the short proof using the deletion method.
 Choose $k$ vectors $x_1,\dots,x_k$ independently from the unit ball of $\bR^d$.
 For any two $x_i$ and $x_j$ of them, the probability that $\norm[2]{x_i-x_j}\leq 2/\sqrt{3}$ is at most $(\sqrt{8}/3)^d =: p$.
 Similarly, the probability that $\norm[2]{x_i+x_j}\leq 2/\sqrt{3}$ is at most $(\sqrt{8}/3)^d =: p$.
Thus the expected number of pairs $\{x_i,x_j\}$ such that $\norm[2]{x_i-x_j}\leq 2/\sqrt{3}$ or $\norm[2]{x_i+x_j}\leq 2/\sqrt{3}$ is at most $2p \binom{k}{2}$ by the union bound.
It follows that there exist $x_1,\dots,x_k$ from the unit ball such that less than $pk^2$ pairs satisfy $\norm[2]{x_i\pm x_j}\leq 2/\sqrt{3}$.
 For each one of these pairs of vectors, remove one of them.
 There remains at least $k - pk^2$ vectors such that for any pair $x_i$ and $x_j$ we have $\norm[2]{x_i\pm x_j} > 2/\sqrt{3}$.
 If we normalise each $x_i$, we obtain at least $k-pk^2$ unit vectors such that all pairwise inner product are in $(-1/3,1/3)$.
 This is maximised by choosing $k = 1/(2p)$, which proves the theorem.
\end{proof}

To prove Proposition~\ref{thm:uppersmall2}, we need the following counterpart to Proposition~\ref{prop:exist}.
\begin{proposition}\label{prop:exist2}
If there exists a $d$-dimensional smooth and strictly convex $K$ with separable Hadwiger number of at least $n$, then for some $k\geq 0$ there exists an $(n-2k)\times (n-2k)$ matrix $[a_{ij}]$ of rank at most $d-k+1$ and satisfying $a_{ii}=1$ for all $i=1,\dots,n-2k$, and $-1/3 < a_{ij} < 1/3$ for all distinct $i,j$.
\end{proposition}
\begin{proof}
Consider the vectors $x_i\in\bR^d$ and functionals $\fhi_i\in(\bR^d)^*$, $i=1,\dots n$ given by Lemma~\ref{lem:BN18}.
Whenever there exist distinct $i,j$ such that $\fhi_i(x_j)=-1$, we must have $x_i=-x_j$ and $\fhi_i=-\fhi_j$.
It follows that $\fhi_i(x_k)=0$ for all $k\neq i,j$, and thus we can remove $x_i$, $x_j$, $\fhi_i$, $\fhi_j$ and move one dimension down.
We repeat this $k$ times until all $-1<\fhi_i(x_j)$ for all distinct $i,j$, and then $[\fhi_i(x_j)]_{i,j}$ is an $(n-2k)\times (n-2k)$ matrix $[a_{ij}]$ of rank at most $d-k$, with $1$s on the diagonal and all other entries in the interval $(-1,0]$.
Choose $\epsi\in(0,1)$ sufficiently small such that $-1+\epsi\leq \fhi_i(x_j)$ whenever $i\neq j$.
If we set $a_{ij}=((2+\epsi)\fhi_i(x_j)+1-\epsi)/3$, then $[a_{ij}]$ is the desired matrix.
\end{proof}

We recall a well-known result on the rank of a square matrix in terms of its trace and Frobenius norm.
\begin{lemma}[Rank Lemma]
The rank $r$ of a real non-zero square matrix $A=[a_{ij}]$ satisfies $r\geq (\sum_{i}a_{ii})^2/\sum_{i,j}a_{ij}^2$.
\end{lemma}

\begin{proof}[Proof of Proposition~\ref{thm:uppersmall2}]
 Assume there is a smooth and strictly convex body with separable Hadwiger number $n$ in $\Red$ with $d=5,6$ or 7. Let $A$ be the $m\times m$ matrix, $m=n-2k$, of rank $r\leq d-k+1$ from Proposition~\ref{prop:exist2}.
Now, the Rank Lemma yields
 \[
  d-k+1\geq r > \frac{m^2}{m+(m^2-m)/9}=\frac{9m}{m+8},
 \]
hence $n < \frac{8(d-k+1)}{8-d+k}+2k$.
It can be checked that this last expression is maximised when $k=0$ since $d\geq 4$.
Thus, $n < \frac{8(d+1)}{8-d}$, and Proposition~\ref{thm:uppersmall2} follows.
\end{proof}

Next, we turn to the proof of Theorem~\ref{thm:ell1}.
Alon \cite{Alon97} uses a Reed--Solomon code to construct a packing of unit 
Euclidean balls in $\bR^d$ such that each ball touches at least $2^{\sqrt{d}}$ 
others when $d=4^k$ for some $k\in\bN$. 
This code $C\subseteq\set{0,1}^d$ has the property that any two points in $C$ have 
Hamming distance at least some fixed $D$, and for any point in $C$, there are at least $2^{\sqrt{d}}$ points of $C$ at Hamming distance $D$.
It is easily seen that using the same code, this 
construction works for any \emph{unconditional convex body} (symmetric with respect to all the coordinate hyperplanes), for instance the 
$\ell_p^d$-unit balls.
In the next proposition, we show that for $\ell_1^d$-unit balls this construction gives a totally 
separable packing. This implies Theorem~\ref{thm:ell1}.

\begin{proposition}\label{prop:ellonepacking}
Let $C\subseteq\set{0,1}^d$ be such that any two vectors from $C$ have 
Hamming distance at least $D$.
Then the packing of $\ell_1^d$ balls of radius $D/2$ with centres from $C$ is 
totally separable.
\end{proposition}
\begin{proof}
We show the stronger property that for any $u\in C$, the $\ell_1^d$-ball with 
centre $u$ and radius $D/2$ has a supporting hyperplane $H$ such that all 
$\ell_1^d$-balls with centres from $C\setminus\set{u}$ and radius $D/2$ lie in 
the closed half-space bounded by $H$ on the opposite side of $u$.

Define the linear functional $f = 2u - 
(1,1,\dots,1)\in(\ell_1^d)^*=\ell_\infty^d$.
Then $\|f\|_\infty=1$ and $f(u)$ equals the number of $1$s in the coordinates of $u$.

Next consider any $u'\in C\setminus\set{u}$.
For each $i,j\in\set{0,1}$, let $m_{ij}$ denote the number of coordinates 
$t\in\set{1,\dots,d}$ where $u_t=i$ and $u'_t=j$.
Then 
the Hamming distance between $u$ and $u'$ is $m_{01}+m_{10}\geq D$, 
$f(u)=m_{10}+m_{11}$, and $f(u')= m_{11}-m_{01}$.
It follows that $f(u)-f(u') = m_{01}+m_{10}\geq D$, and since $\norm[\infty]{f}=1$, 
the hyperplane $H=\setbuilder{x\in\bR^d}{f(x)=f(u)-D/2}$ separates the $\ell_1^d$ balls around $u$ and $u'$ of 
radius~$D/2$.
Indeed, for any $x\in\ell_1^d$ satisfying $\norm[1]{x-u}\leq D/2$
we have $f(x-u)\geq -D/2$, hence $f(x)\geq f(u)-D/2$,
and for any $x\in\ell_1^d$ satisfying $\norm[1]{x-u'}\leq D/2$
we have $D/2\geq f(x-u')=f(x)-f(u)+f(u)-f(u')\geq f(x)-f(u)+D$, hence $f(x)\leq f(u)-D/2$.
\end{proof}

\section{Preliminaries on planar packings}\label{sec:preliminaries}
It will be helpful to note that quasi hexagons are those $o$-symmetric planar convex bodies that have an affine position between a certain affine regular hexagon and a square, in the following sense.
We leave the proof to the reader.
\begin{lemma}\label{lem:quasi}
An $o$-symmetric convex body $K$ in $\bR^2$ is a quasi hexagon iff there is a linear $T\colon\bR^2\to\bR^2$ such that $H\subseteq T(K)\subseteq S$, where $H=\conv\set{(\pm1,0),(0,\pm 1),\pm(1,1)}$ and $S=[-1,1]^2$. \qed
\end{lemma}

\begin{lemma}\label{lem:triangle}
Let $K$ be an $o$-symmetric convex body.
Let $u_1,u_2\in\bd K$ be such that $\{K,K+2u_1,K+2u_2\}$ is a totally separable packing in which any two of the translates touch.
Then $K$ is a quasi hexagon.
\end{lemma}
\begin{proof}
By possibly relabelling the three translates, we can assume that there exist two intersecting lines $\ell_1$ and $\ell_2$ such that $K$ and $K+2u_i$ are on opposite sides of $\ell_i$, and $K$ and $K+2u_{1-i}$ on the same side of $\ell_i$ ($i=1,2$).
Since $K+2u_1$ and $K+2u_2$ are separated by both $\ell_1$ and $\ell_2$, it follows that the intersection point of the two lines equals $u_1+u_2$, and that $\bd K$ contains the segments $[u_1,u_1+u_2]$ and $[u_2,u_1+u_2]$.
Thus, $K$ is a quasi hexagon.
\end{proof}

Since the contact graph of a totally separable packing of translates of $K$ is planar, as long as $K$ is not a parallelogram \cite{Br96}, and Lemma~\ref{lem:triangle} implies that the contact graph is triangle-free if $K$ is not a quasi hexagon, we can already deduce that the number of edges in a contact graph is at most $2n-4$.
In the next section, we improve this bound to $2n-2\sqrt{n}$ with the use of a certain angular measure, which we now introduce.

\begin{definition}
For two points $p,q$ on $\bd K$, we denote the \emph{clockwise arc} of $\bd K$ connecting $p$ and $q$ by $\darc{pq}$.
If $p\neq -q$, then we denote the \emph{minor arc} connecting $p$ and $q$ by $\arc{pq}=\arc{qp}$.
\end{definition}

\begin{definition}\label{def:angularmeasure}
Let $K$ be an $o$-symmetric convex body on the plane. A Borel measure $\mu$ on 
$\bd K$ is an \emph{angular measure} if
\begin{enumerate}
\item $\mu(\bd K)=2\pi$,
\item $\mu(X)=\mu(-X)$ 
for every Borel subset $X$ of $\bd K$, and
\item $\mu$ is \emph{continuous}, that is, 
$\mu(\set{p})=0$ for every $p\in \bd K$.
\end{enumerate}
\end{definition}

We note that the angular measure of any arc connecting two opposite points of $\bd K$ (ie., a 
\emph{semicircle}) equals $\pi$.
It follows that for any angular measure, the sum of the interior angles of a triangle equals $\pi$.
This can be shown in exactly the same way as the proof for the Euclidean angular measure.
It then follows by triangulation (see \cite{Br96}) that for any simple closed polygon 
with $v$ vertices, 
\begin{equation}\label{eq:anglesumpolygon}
\text{the sum of the interior angles }  = (v-2)\pi.
\end{equation}
An $o$-symmetric convex body $K$ has many angular measures.
For example, we can take the Euclidean angular measure after choosing a Euclidean structure on the plane, or we can take the one-dimensional Hausdorff measure on $\bd K$, normalized to $2\pi$.
Most of these angular measures have no further interesting properties beyond \eqref{eq:anglesumpolygon}, and where angular measures have been used in the literature, specific ones with further properties have been constructed.
Brass \cite{Br96} constructed an angular measure in which the angles of any equilateral triangle are all $\pi/3$.
In \cite{BKO} the notion of a \emph{B-measure} is used. This is an angular 
measure $\mu$ such that $\mu(A)=\pi/2$ whenever $A$ is an arc on $\bd K$ with 
endpoints $a$ and $b$, say, such that $a$ is \emph{Birkhoff orthogonal} to $b$, that 
is, the line through $a$ in the direction $b$ supports $K$ at $a$. It turns out 
that B-measures exist only in very special cases \cite{NPS20}.
In \cite{BKO}, Theorem~\ref{thm:csep} is proved for all smooth and strictly convex bodies by approximating these bodies with bodies admitting a B-measure.
We avoid the need for approximation by using the following type of angular measure that exists for any $K$ that is not a quasi hexagon.
\begin{definition}
Let $K$ be an $o$-symmetric convex body in the plane that is not a quasi hexagon.
A \emph{$\pi$-measure} is an angular measure $\mu$ on $\bd K$ if for every segment $[a,b]\subset\bd K$ with $\norm[K]{a-b} > 1$ we have
\begin{enumerate}
\item $\mu(\arc{bc})=0$ where $c\in [a,b]$ is such that $\norm[K]{a-c}=1$, and $\mu(\arc{ad})=0$ where $d\in[a,b]$ is such that $\norm[K]{b-d}=1$,
\item for any segment $[a,e]\subset\bd K$ not parallel to $[a,b]$, we have $\mu(\arc{ae})=0$.
\end{enumerate}
\end{definition}

Any segment $[a,e]$ in the above definition will have $K$-length strictly smaller than $1$, since $K$ is not a quasi hexagon.
It is easy to see that any $K$ that is not a quasi hexagon admits a $\pi$-measure, since after removing the arcs that are required to have measure~$0$ from $\bd K$, there will always be some arcs left on which a mass of $2\pi$ can be distributed.
The following property of $\pi$-measures will be crucial for the proof of Theorem~\ref{thm:csep}.

\begin{lemma}\label{lem:piangle}
Let $K$ be an $o$-symmetric convex body that is not a quasi hexagon, and let $\mu$ be a $\pi$-measure on $\bd K$.
Let $u_0, u_1, u_2\in\bd K$ be such that $u_1$ is on the directed arc $\darc{u_2 u_0}$ and either $\{K,K+2u_0,K+2u_1,K+2u_2\}$ or $\{K,K+2u_0,K+2u_1,K+2(u_1-u_2)\}$ is a totally separable packing.
Then $\mu(\darc{u_2 u_0})\geq\pi$.
\end{lemma}

\begin{proof}
We first consider the totally separable packing $\{K,K+2u_0,K+2u_1,K+2u_2\}$.
Let $\ell_i$ be a line separating $K$ and $K+2u_i$ that does not intersect the interiors of the other two translates ($i=0,1,2$).
If $K+2u_0$ and $K+2u_2$ are both on the side of $\ell_1$ opposite $K+2u_1$, then $\darc{u_2 u_0}$ contains a semicircle, hence $\mu(\darc{u_2 u_0})\geq\pi$.
Thus we assume without loss of generality that $\mu(\darc{u_2 u_0}) < \pi$ and that $K+2u_1$ and $K+2u_2$ are on the same side of $\ell_1$ (Figure~\ref{fig:firstcase}).

\begin{figure}[ht]\centering
  \begin{tikzpicture}[scale=0.27]
  \draw (-5,-6) -- (-5,20) node[below right] {$\ell_0$};
  \draw (5,-6) -- (5,17) node[below right] {$\ell_2$};
  \draw (-9.5,5) -- (14,5) node[right] {$\ell_1$};
  
  \def\drawK{\draw[draw=blue, fill=blue!10!white]
	(-1,-5) -- (5,-5) -- (5,-3) .. controls (4,2) and (4,3) ..
	(1,5) -- (-5,5) -- (-5,3) .. controls (-4,-2) and (-4,-3) .. cycle;}

  \begin{scope}[thick, blue, scale=0.99]
  \drawK
  \draw[blue!50!black] (0,0) node {$K+2u_1$};
  \end{scope}

  \begin{scope}[thick, xshift=10cm, blue, scale=0.99]
  \drawK
  \draw[blue!50!black] (0,0) node {$K+2u_2$};
  \end{scope}

  \begin{scope}[xshift=-10cm, yshift=16cm, thick, blue, scale=0.99]
  \drawK
  \draw[blue!50!black] (0,0) node {$K+2u_0$};
  \end{scope}

  \begin{scope}[xshift=0cm, yshift=10cm, thick, blue, scale=0.99]
  \drawK
  \draw[blue!50!black] (0,2) node {$K$};
  \fill[blue!50!black] (0,0) circle (5pt);
  \draw[blue!50!black] (-0.7,-0.5) node{$o$};
  \draw[-latex] (0,0) -- (-5,3);
  \draw[blue!50!black] (-4,3.7) node{$u_0$};
  \draw[-latex] (0,0) -- (0,-5);
  \draw[blue!50!black] (1,-4) node{$u_1$};
  \draw[-latex] (0,0) -- (5,-5);
  \draw[blue!50!black] (6,-4.2) node{$u_2$};
  \draw[-latex, blue!50!white] (0,0) -- (5,-3);
  \draw[blue!50!white] (6.6,-2.7) node{$-u_0$};
  \end{scope}
  \end{tikzpicture}
\caption{Lemma~\ref{lem:piangle}, when $\{K,K+2u_0,K+2u_1,K+2u_2\}$ is a totally separable packing.}\label{fig:firstcase}
\end{figure}
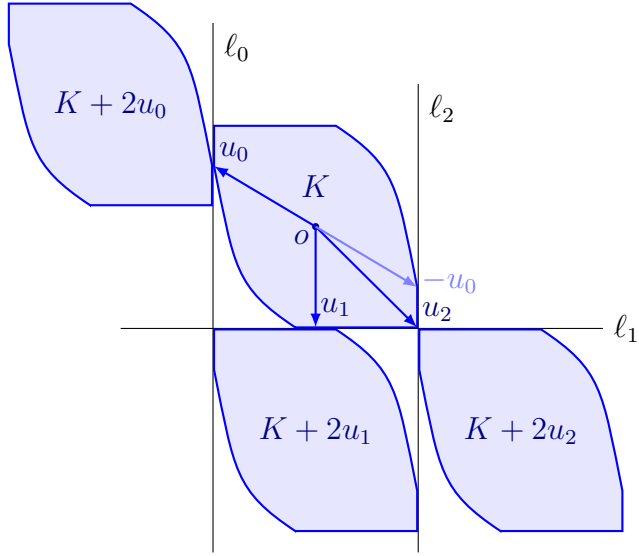

\begin{figure}[ht]\centering
  \begin{tikzpicture}[scale=0.27]
  \draw (-17,5) -- (17,5) node[right] {$\ell_1$};
  \draw (-17,15) -- (17,15) node[right] {$\ell_1-2u_1$};

  \def\drawK{\draw[draw=blue, fill=blue!10!white]
	(-5,-5) -- (5,-5) -- (5,5) -- (-5,5) -- (-5,-5);}

  \begin{scope}[thick, xshift=-10cm, blue, scale=0.99]
  \drawK
  \draw[blue!50!black] (0,0) node {$P+2u_0$};
  \end{scope}

  \begin{scope}[thick, xshift=10cm, blue, scale=0.99]
  \drawK
  \draw[blue!50!black] (0,0) node {$P+2u_1$};
  \end{scope}

  \begin{scope}[xshift=0cm, yshift=10cm, thick, blue, scale=0.99]
  \drawK
  \draw[blue!50!black] (0,2) node {$P$};
  \fill[blue!50!black] (0,0) circle (5pt);
  \draw[blue!50!black] (-0.7,-0.5) node{$o$};
  \draw[blue!50!black] (6,-4.2) node{$u_1$};
  \draw[blue!50!black] (-6,-4.2) node{$u_0$};
  \draw[blue!50!black] (-6,6) node{$-u_1$};
  \draw[blue!50!black] (6,6) node{$-u_0$};
  \end{scope}
  \end{tikzpicture}
\caption{Lemma~\ref{lem:piangle}, Case~2}\label{fig:casetwo}
\end{figure}
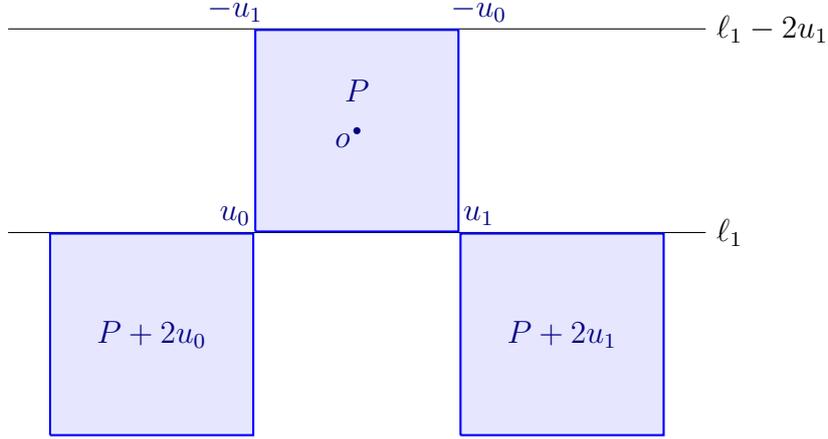

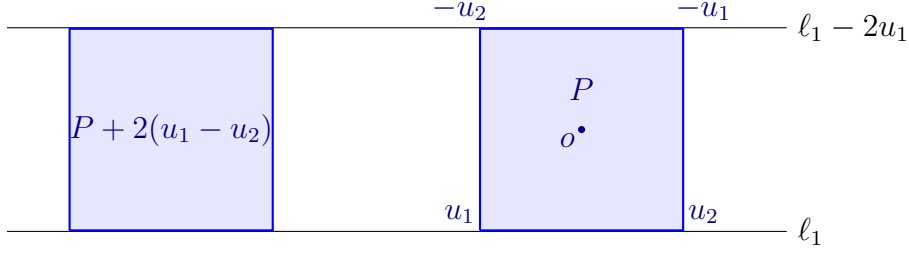
\begin{figure}[ht]\centering
  \begin{tikzpicture}[scale=0.27]
  \draw (-28,5) -- (10,5) node[right] {$\ell_1$};
  \draw (-28,15) -- (10,15) node[right] {$\ell_1-2u_1$};

  \def\drawK{\draw[draw=blue, fill=blue!10!white]
	(-5,-5) -- (5,-5) -- (5,5) -- (-5,5) -- (-5,-5);}

  \begin{scope}[thick, xshift=-20cm, yshift=10cm, blue, scale=0.99]
  \drawK
  \draw[blue!50!black] (0,0) node {$P+2(u_1-u_2)$};
  \end{scope}

  \begin{scope}[xshift=0cm, yshift=10cm, thick, blue, scale=0.99]
  \drawK
  \draw[blue!50!black] (0,2) node {$P$};
  \fill[blue!50!black] (0,0) circle (5pt);
  \draw[blue!50!black] (-0.7,-0.5) node{$o$};
  \draw[blue!50!black] (6,-4.2) node{$u_2$};
  \draw[blue!50!black] (-6,-4.2) node{$u_1$};
  \draw[blue!50!black] (-6,6) node{$-u_2$};
  \draw[blue!50!black] (6,6) node{$-u_1$};
  \end{scope}
  \end{tikzpicture}
\caption{Lemma~\ref{lem:piangle}, Case~3}\label{fig:casethree}
\end{figure}

Since $K$ touches $K+2u_1$ and $K+2u_2$ in $u_1$ and $u_2$, respectively, $\bd K$ contains the segment $[u_1,u_2]$.
Since $\mu(\darc{u_2 u_0}) < \pi$, we have that $\darc{u_2 u_0}$ is a minor arc, and the separating lines $\ell_0$ and $\ell_2$ are either parallel or intersect on the same side as $u_1$ of the line $ou_0$.
If $\ell_0$ and $\ell_2$ intersect, then no translate of $K$ can lie in the triangle bounded by $\ell_0, \ell_1, \ell_2$.
It follows that $\ell_0$ and $\ell_2$ are parallel and hence $[-u_0,u_2]\subset \bd K$.
Therefore, $\mu(\darc{-u_0,u_2})=0$ by definition of $\pi$-measure, and we conclude that $\mu(\darc{u_2 u_0})=\pi$.

Next consider the totally separable packing $\{K,K+2u_0,K+2u_1,K+2(u_1-u_2)\}$.
Let $\ell_1$ be a line separating $K$ and $K+2u_1$ that does not intersect the interiors of the other two translates. We will say that a translate of $K$ is \emph{above} (resp., \emph{below}) $\ell_1$ if it lies on the same (resp., opposite) side of $\ell_1$ as $K$. We consider three cases.

\emph{Case~1:} Assume that $K+2u_0$ is above and $K+2(u_1-u_2)$ is below $\ell_1$. Then both $u_0$ and $u_2$ lie on the side of $\ell_1-u_1$ opposite to $u_1$, hence $\darc{u_2u_0}$ contains a semicircle, and the claim follows.

\emph{Case~2:} Assume that $K+2u_0$ is below $\ell_1$. 
Then $[u_0,u_1]\subset\bd K\cap\ell_1$ and, by Lemma~\ref{lem:triangle}, $\norm[K]{u_1-u_0}>1$.
Let $P$ denote the parallelogram $\conv\{\pm u_0,\pm u_1\}$ 
inscribed to $K$ with $P\neq K$. Moreover, $\ell_1$ and $\ell_1-2u_1$ are support lines of $K$. Since no translate of $K$ `fits between' $P+2u_0$ and $P+2u_1$, it follows that $u_2\in\ell_1-2u_1$ (Figure~\ref{fig:casetwo}).

Assuming that $\darc{u_2u_0}$ does not contain a semicircle, the points $-u_1,-u_0,u_2$ follow each other on $\bd K\cap(\ell_1-2u_1)$ in this clockwise order. Since $\norm[K]{u_1-u_0}>1$, we have that $\mu(\darc{-u_0,u_2})=0$ and hence, $\mu(\darc{u_2u_0})=\mu(\darc{-u_0,u_0})=\pi$ as required.

\emph{Case~3:} Assume that both $K+2u_0$ and $K+2(u_1-u_2)$ are above $\ell_1$.
We can use a very similar argument to the previous case.
We have that $[u_2,u_1]\subset\bd K\cap\ell_1$ and, by Lemma~\ref{lem:triangle}, $\norm[K]{u_2-u_1}>1$. Let $P$ denote the parallelogram $P=\conv([u_2,u_1]\cup[-u_2,-u_1])$. Clearly, $P\subset K$ with $P\neq K$. Moreover, $\ell_1$ and $\ell_1-2u_1$ are support lines of $K$. Since no translate of $K$ `fits between' $P$ and $P+2(u_1-u_2)$, it follows that $u_0\in\ell_1-2u_1$, see Figure~\ref{fig:casethree}.

Assuming that $\darc{u_2u_0}$ does not contain a semicircle, the points $u_0,-u_2,-u_1$ follow each other on $\bd K\cap(\ell_1-2u_1)$ in this clockwise order. Since $\norm[K]{u_2-u_1}>1$, we have that $\mu(\darc{u_0,-u_2})=0$ and hence, $\mu(\darc{u_2u_0})=\mu(\darc{u_2,-u_2})=\pi$ as required, completing the proof of Lemma~\ref{lem:piangle}.

\end{proof}

Lemma~\ref{lem:piangle} can now be applied to find a simple derivation of the totally separable Hadwiger numbers of bodies that are not quasi hexagons.

\begin{lemma}\label{lem:hsep}
 Let $K$ be an $o$-symmetric convex body in the plane.
 \begin{enumerate}[(i)]
  \item\label{item:paral} If $K$ is a parallelogram, then $\Hsep(K)=8$.
  \item\label{item:hex} If $K$ is a quasi hexagon but not a parallelogram, then 
$\Hsep(K)=6$.
  \item\label{item:nothex} If $K$ is not a quasi hexagon, then $\Hsep(K)=4$.
 \end{enumerate}
\end{lemma}

\begin{proof}
It is well known \cite{Gru61} that the Hadwiger number of a parallelogram is 8, 
and that of any 
other planar convex body is 6, thus giving the upper bounds in 
(\ref{item:paral}) and (\ref{item:hex}). The lower bounds are shown by 
Figure~\ref{fig:totsep}.
\begin{figure}
\captionsetup[subfigure]{font=footnotesize}
\centering
\subcaptionbox{parallelogram}[.3\textwidth]{%
\begin{tikzpicture}[scale=0.1]
  \draw (-15,-19) -- +(0,38);
  \draw (-5,-19) -- +(0,38);
  \draw (5,-19) -- +(0,38);
  \draw (15,-19) -- +(0,38);
  \draw (-19,-15) -- +(38,0);
  \draw (-19,-5) -- +(38,0);
  \draw (-19,5) -- +(38,0);
  \draw (-19,15) -- +(38,0);
  \def\drawK{\draw[draw=blue, fill=blue!10!white]
	(-5,-5) -- (5,-5) -- (5,5) -- (-5,5) -- cycle;}
  \foreach \x in {-10cm,0cm,10cm} {
      \foreach \y in {-10cm,0cm,10cm} {
	  \begin{scope}[xshift = \x, yshift = \y, thick, scale=0.99]
	      \drawK
	  \end{scope}
      }
  }
  \begin{scope}[thick, scale=0.99]
  \draw[draw=red, fill=red!20!white]
	(-5,-5) -- (5,-5) -- (5,5) -- (-5,5) -- cycle;
  \end{scope}
\end{tikzpicture}}%
\subcaptionbox{quasi hexagon}[.4\textwidth]{
  \begin{tikzpicture}[xslant=-0.57,yscale=0.9,scale=0.11]
\draw (-15,-19) -- +(0,38);
\draw (-5,-19) -- +(0,38);
\draw (5,-19) -- +(0,38);
\draw (15,-19) -- +(0,38);
\draw (-19,-15) -- +(38,0);
\draw (-19,-5) -- +(38,0);
\draw (-19,5) -- +(38,0);
\draw (-19,15) -- +(38,0);
\def\drawK{\draw[draw=blue, fill=blue!10!white]
      (-5,-5) -- (0,-5) .. controls (1,-5) and (5,-1) .. (5,0) -- (5,5) -- 
(0,5) .. controls (-1,5) and (-5,1) .. (-5,0) -- cycle;}
\foreach \x in {-10cm,0cm,10cm} {
    \foreach \y in {-10cm,0cm,10cm} {
        \begin{scope}[xshift = \x, yshift = \y, thick, blue, scale=0.98]
            \drawK
        \end{scope}
    }
}
\begin{scope}[thick, scale=0.98]
\draw[draw=red, fill=red!20!white]
      (-5,-5) -- (0,-5) .. controls (1,-5) and (5,-1) .. (5,0) -- (5,5) -- 
(0,5) .. controls (-1,5) and (-5,1) .. (-5,0) -- cycle;
\end{scope}

\end{tikzpicture}}%
\subcaptionbox{non-quasi hexagon}[.3\textwidth]{%
\begin{tikzpicture}[scale=0.1]
  \draw (-15,-19) -- +(0,38);
  \draw (-5,-19) -- +(0,38);
  \draw (5,-19) -- +(0,38);
  \draw (15,-19) -- +(0,38);
  \draw (-19,-15) -- +(38,0);
  \draw (-19,-5) -- +(38,0);
  \draw (-19,5) -- +(38,0);
  \draw (-19,15) -- +(38,0);
  \def\drawK{\draw[draw=blue, fill=blue!10!white]
	(0,0) circle [radius=5cm];}
  \foreach \x in {-10cm,0cm,10cm} {
      \foreach \y in {-10cm,0cm,10cm} {
	  \begin{scope}[xshift = \x, yshift = \y, thick, scale=0.99]
	      \drawK
	  \end{scope}
      }
  }
  \begin{scope}[thick, scale=0.99]
  \draw[draw=red, fill=red!20!white]
	(0,0) circle [radius=5cm];
  \end{scope}
\end{tikzpicture}}%
\caption{Totally separable packings}\label{fig:totsep}
\end{figure}
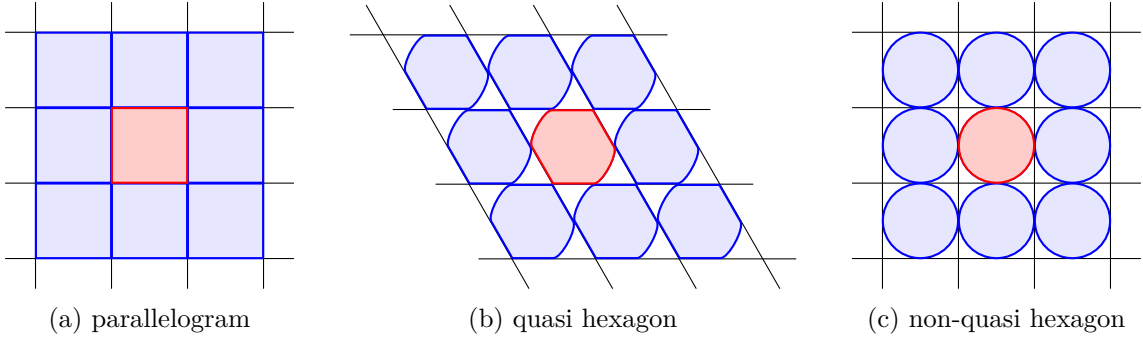

We prove the third statement.
First note that any $o$-symmetric convex body $K$ has a circumscribed parallelogram such that the midpoints of the sides of the parallelogram are on the boundary of $K$.
If we let the midpoints be $p_1,p_2,p_3,p_4$, then $\{K\}\cup\setbuilder{K+2p_i}{i=1,\dots,4}$ is a totally separable packing, 
hence $\Hsep(K)\geq 4$.

Next we assume that $K$ is not a quasi hexagon, and show that $\Hsep(K)\leq 4$.
Suppose for a contradiction that $q_0,q_1, \ldots,q_4\in\bd K$ are unit vectors such that 
$\{K, 2q_0+K,\ldots,2q_4+K\}$ is a totally separable packing.
We may assume that $q_0,q_1,\dots,q_4$ are in clockwise order on $\bd K$.
Then by Lemma~\ref{lem:piangle}, for any $\pi$-measure $\mu$ on $\bd K$ we have $\mu(\arc{q_i q_{i+2}})\geq\pi$ for each $i=0,\dots,4$, where subscripts are taken modulo~$5$.
If we add all these angles, we obtain that $2\mu(\bd K)=5\pi$, a contradiction.
This completes the proof of Lemma~\ref{lem:hsep}.
\end{proof}

Although we will not use it, it can be shown in a similar way that if $K$ is not a quasi hexagon, any two 
vertices in the contact graph of a totally separable packing of translates of $K$ share at most two neighbors.

\section{Proof of Theorem~\ref{thm:csep}}\label{sec:csep}

We follow the 
line of the argument of the proof of \cite[Theorem~13]{BKO}, where we use a $\pi$-measure on the boundary of a non-quasi hexagon.
We also fix a gap in the proof that also occurs in the proof of the Euclidean case in \cite{BSS16} which goes back to \cite{H74}.

We show the upper bound by induction on $n$. Let $G$ be the contact graph of a totally 
separable packing of $n$ translates of $K$ with the maximum number $e=\csep(K,n)$ of 
edges.
The base cases $n\leq4$ follow immediately from the fact that $G$ is 
triangle-free (Lemma~\ref{lem:triangle}).

We may assume that $G$ is $2$-connected.
Indeed, suppose that by the removal of a 
vertex $x$, $G$ becomes disconnected.
Then $G$ can be decomposed into two subgraphs $G_1$ and $G_2$ with only the vertex $x$ in common and with no common edges.
Denote the number of vertices of $G_i$ by $n_i\geq 2$, $i=1,2$, where we assume $n_1\leq n_2$ without loss of generality.
Then $n_1+n_2=n+1$, and by the induction hypothesis, the number of edges of $G$ is at most $\lfloor 2n_1-2\sqrt{n_1}\rfloor+\lfloor 2n_2-2\sqrt{n_2}\rfloor$.
It is easy to show that this is at most $\lfloor 2n-2\sqrt{n}\rfloor$.
Indeed, if $4\leq n_1\leq n_2$, then by the convexity of the function $t\mapsto 2(t-\sqrt{t})$, we have that $2n_1-2\sqrt{n_1}+2n_2-2\sqrt{n_2}\leq 2\cdot 4-2\sqrt{4} + 2(n-3)-2\sqrt{n-3}$, and this can be checked to be at most $2n-2\sqrt{n}$ for all $n\geq 4$.
If $n_1=3\leq n_2$, then $\lfloor 2n_1-2\sqrt{n_1}\rfloor+\lfloor 2n_2-2\sqrt{n_2}\rfloor \leq 2 + 2(n-2)-2\sqrt{n-2}\leq 2n-\sqrt{2n}$.
The case $n_1=2$ is similar.

We can avoid these calculations by using the fact that in a polyomino with $n$ cells, the number of pairs touching in an edge is at most $\lfloor 2n-\sqrt{2n}\rfloor$ (see \cite{BKO} and references therein).
Consider any polyominoes $P_1$ and $P_2$ with exactly $\lfloor 2n_i-2\sqrt{n_i}\rfloor$ touching cells, respectively.
It is then sufficient to translate $P_2$ so that one of its cells overlaps with a cell of $P_1$ without there being any other overlaps.
This can be done by moving the left-most cell in the top row of $P_2$ so that it overlaps with the right-most cell in the bottom row of $P_1$.
This gives a polyomino with $n$ cells and at least $\lfloor 2n_1-2\sqrt{n_1}\rfloor+\lfloor n_2-2\sqrt{n_2}\rfloor$ cells, which is bounded above by $\lfloor 2n-\sqrt{2n}\rfloor$.

Thus, $G$ is a $2$-connected planar graph on $n$ vertices, each of 
degree $2$, $3$ or $4$ by Lemma~\ref{lem:hsep}.
Consider the closed polygon $P$ with vertices $a_1,\ldots,a_v$ 
bounding the outer face of $G$ in counter-clockwise orientation.
We will always refer to the indices modulo 
$v$.
Note that $\norm[K]{a_i-a_{i+1}}=2$ for each $i$.

We fix a $\pi$-measure $\mu$ on $\bd K$.
Let $v_2, v_3$ and $v_4$ denote the number of vertices of degree $2$, $3$ and $4$ of 
$P$, where $v_2+v_3+v_4=v$.
Each vertex of $P$ is also on some bounded faces of $G$, and for each such bounded face there is an angle at that vertex.
In total, there are $v_2+2v_3+3v_4$ such angles inside $P$ with vertex on $P$.
We call two such angles \emph{adjacent} if either they have the same vertex and share an edge, or their vertices are adjacent on $P$ and they share the same bounded face.
By Lemma~\ref{lem:piangle}, the sum of the $\mu$-measure of two adjacent angles is $\geq\pi$.
Thus if we sum up the measures of all adjacent pairs of angles, we obtain a sum of at least $\pi(v_2+2v_3+3v_4)$.
Since each angle blongs to two such pairs, we sum each angle twice.
By \eqref{eq:anglesumpolygon}, we obtain $2\pi(v-2)\geq\pi(v_2+2v_3+3v_4)$.
From this we obtain
\begin{equation}\label{eq:anglesum}
 v_2+2v_3+3v_4\leq 2v-4.
\end{equation}

For any positive integer $j$, denote by $g_j$ the number of internal faces of 
$G$ with $j$ edges. Since $G$ is triangle-free, $g_1=g_2=g_3=0$. Since $G$ has 
$e=\csep(K,n)$ edges, Euler's formula gives $n-e+g_4+g_5+\ldots=1$. If we 
add the number of edges bounding the internal faces, then internal edges are 
counted twice, edges of $P$ are counted once, so we have 
$4(1-n+e)=4(g_4+g_5+\ldots)\leq 4g_4+5g_5+\ldots=v+2(e-v)$. 
This yields 
\begin{equation}\label{eq:smallbound1}
 e\leq 2n-2-v/2.
\end{equation}
If $v\geq 4\sqrt{n}-4$, then we obtain $e\leq 2n-2\sqrt{n}$, and we are done.
Thus we assume without loss of generality that $v < 4\sqrt{n}-4$.
We delete from $G$ the $v$ vertices of $P$ and all edges incident to them. If there is no edge in $G$ that joins two non-adjacent vertices of $P$ (a \emph{diagonal}), we remove exactly $v_2+2v_3+3v_4$ edges from $P$, and we 
obtain
\[e-(v_2-2v_3+3v_4)\leq 2(n-v)-2\sqrt{n-v}\]by the induction hypothesis.
Then \eqref{eq:anglesum} yields
\begin{equation*}
 e\leq 2n-4-2\sqrt{n-v} < 2n-4-2\sqrt{n-4\sqrt{n}+4} = 2n-4-2(\sqrt{n}-2) = 2n-2\sqrt{n},
\end{equation*}
noting that $n\geq 4$.
To finish the proof of Theorem~\ref{thm:csep}, it remains to consider the case where $P$ has a diagonal.
In this case there exist two contact graphs $G_1$ and $G_2$ with union $G$, intersecting in just the diagonal and its two endpoints, such that if $G_i$ has $n_i\geq 4$ vertices and $e_i$ edges ($i=1,2$), then $n_1+n_2=n+2$ and $e_1+e_2=e+1$.
By induction we obtain
\begin{align*}
e &=e_1+e_2-1\leq 2n_1-2\sqrt{n_1} + 2n_2-2\sqrt{n_2} -1 = 2n+3-2(\sqrt{n_1} + \sqrt{n_2})\\
&\leq 2n+3-2(\sqrt{4}+\sqrt{n-2}) \quad\text{by convexity}\\
&< 2n-2\sqrt{n}\quad\text{since $n\geq 6$}.
\end{align*}

\section{Final Remarks}\label{sec:final}

The following problems remain open.

\begin{question}
In the proof of Proposition~\ref{prop:ellonepacking}, we only use $2^{\sqrt{d}}$ 
separating hyperplanes around each point. Can we modify the $\ell_1^d$-unit ball 
so that it becomes smooth or strictly convex? 
\end{question}
 
\begin{question}
Can we improve the $2^{d+1}-3$ upper bound for the totally separable Hadwiger 
number of smooth or strictly convex bodies from \cite{BN18}?
\end{question}

\begin{question}
What is the smallest dimension $d$ such that there exists a smooth, strictly convex $d$-dimensional $o$-symmetric convex body $K$ with $H(K) > 2d$?
We have shown that $d\in\{5,6,7,8\}$.
Equivalently, what is the smallest $d$ such that there exists a $(2d+1)\times (2d+1)$ matrix with $1$s on the diagonal, all off-diagonal entries in the interval $(-1,0]$ and rank at most $d$?
\end{question}

\providecommand{\bysame}{\leavevmode\hbox to3em{\hrulefill}\thinspace}
\providecommand{\MR}{\relax\ifhmode\unskip\space\fi MR }
\providecommand{\MRhref}[2]{%
  \href{http://www.ams.org/mathscinet-getitem?mr=#1}{#2}
}
\providecommand{\href}[2]{#2}

\end{document}